\documentclass[a4paper,12pt]{article}
\usepackage{CJKutf8}
\usepackage{CJK}
\usepackage{amsfonts}
\usepackage{pifont}
\usepackage{bm}
\usepackage{latexsym,amsmath,amssymb,cite,amsthm}
\usepackage{color,eucal,enumerate,mathrsfs}
\usepackage[normalem]{ulem}
\usepackage{amsmath}
\usepackage{hyperref}
\usepackage[pagewise]{lineno}

\numberwithin{equation}{section}
\theoremstyle{plain}
\newtheorem{theorem}{Theorem}[section]

\newtheorem{lemma}[theorem]{Lemma}
\newtheorem{proposition}[theorem]{Proposition}
\theoremstyle{definition}
\newtheorem{definition}[theorem]{Definition}
\theoremstyle{remark}
\newtheorem{assumption}[theorem]{Assumption}

\newcommand{\lmt}[2]{\mathop{\lim}_{{#1} \rightarrow {#2}} }

\newcommand{\lip}[1]{{\mathrm{lip}}({#1})}
\newcommand{\esup}[1]{{\mathrm{ess~sup}}{#1}}

\newcommand{\lmti}[2]{\mathop{\underline{\lim}}_{{#1} \rightarrow {#2}} }

\newcommand{\Ric}{{\rm{Ricci}}}

\newcommand{\bRic}{{\bf Ricci}}

\renewcommand{\H}{{\mathrm{H}}}

\newcommand{\Hess}{{\mathrm{Hess}}}

\newcommand{\mm}{\mathfrak m}
\newcommand{\ms}{(X,\d,\mm)}

\newcommand{\rcdkn}{{\rm RCD}^*(K, N)}
\newcommand{\rcd}{{\rm RCD}(K, \infty)}

\newcommand{\M}{\mathbb{M}}
\newcommand{\N}{\mathbb{N}}
\newcommand{\V}{\mathbb{V}}
\newcommand{\R}{\mathbb{R}}

\newcommand{\supp}{\mathop{\rm supp}\nolimits}   %%\newcommand{\span}{\mathop{\rm span}\nolimits}   %
\newcommand{\Lip}{\mathop{\rm Lip}\nolimits}
\renewcommand{\d}{{\mathrm d}}
\newcommand{\dt}{{\d t}}
\newcommand{\ddt}{{\frac \d\dt}}

\newcommand{\D}{{\mathrm D}}
\newcommand{\restr}[1]{\lower3pt\hbox{$|_{#1}$}}
\newcommand{\la}{{\langle}}
\newcommand{\ra}{{\rangle}}

\newcommand{\nchi}{{\raise.3ex\hbox{$\chi$}}}

\title{\large{{\bf New characterizations of Ricci curvature on {\rm RCD} metric measure spaces}}
}
\begin{document}

\author{Bang-Xian Han \thanks
{University of Bonn,  Institute for Applied Mathematics,  han@iam.uni-bonn.de}
}

\date{\today} 
\maketitle

\begin{abstract}
We prove that on  a  large family of  metric measure spaces, if the  $L^p$-gradient estimate for heat flows holds for some $p>2$, then the $L^1$-gradient estimate also holds.   This result extends Savar\'e's result on metric measure spaces, and   provides  a new proof to  von Renesse-Sturm theorem on smooth metric measure spaces. As a consequence, we propose a new analysis object based on Gigli's measure-valued Ricci tensor, to characterize the Ricci curvature of {\rm RCD} space in a local way.

The argument is a new iteration technique  based on  non-smooth Bakry-\'Emery theory,  which is a new method to study the curvature dimension condition of metric measure spaces.
\end{abstract}


\tableofcontents
%%%%%%%%%%%%%%%%%%%%%%%%%%%%
% 前言
%%%%%%%%%%%%%%%%%%%%%%%%%%%
\section{Introduction}
 For any smooth Riemannian manifold $M$ and any $K\in \R$, it is proved by von Renesse and Sturm in \cite{SVR-T} that the following properties are equivalent
\begin{itemize}
\item [1)] $\Ric_M \geq K$,
\item [2)] there exists  $p\in (1, \infty)$ such that for all $f\in C^\infty_c(M)$, all $x\in M$ and $t\geq 0$
 \begin{equation}\label{eq:intro-0}
|\D \H_t f|^p(x) \leq e^{-pKt}\H_t|\D f|^p(x),
\end{equation}
\item [3)] for all $f\in C^\infty_c(M)$, all $x\in M$ and $t\geq 0$
\begin{equation}\label{eq:intro-00}
|\D \H_t f|(x) \leq e^{-Kt}\H_t|\D f|(x),
\end{equation} 
\end{itemize}
where $\H_t f$ is the solution to the heat equation with initial datum $f$.

\bigskip

In non-smooth setting, the notions of synthetic Ricci curvature bounds, or non-smooth curvature-dimension conditions, were proposed by Lott-Villani and Sturm (see \cite{Lott-Villani09} and \cite{S-O1}) using optimal transport theory. Later on, by assuming the infinitesimally Hilbertianity (i.e. the Sobolev space $W^{1,2}$ is a Hilbert space),  {\rm RCD} condition (or $\rcd$  condtion  to emphasize the curvature) which is a refinement of Lott-Sturm-Villani's curvature-dimension condition,  was proposed by Ambrosio-Gigli-Savar\'e (see \cite{AGS-M} and \cite{AGMR-R}). It is known that $\rcd$  spaces are generalizations of Riemannian manifolds with lower Ricci curvature bound and their limit spaces, as well as Alexandrov spaces with lower curvature bound.

Is is known that Lott-Sturm-Villani's synthetic Ricci bound and 2-gradient estimate (for heat flows) are equivalent in non-smooth setting. 
Let $\ms$ be a $\rcd$ space, it is  proved in \cite{AGS-M} that 
\begin{equation}\label{eq:intro-1}
|\D \H_t f|^2 \leq e^{-2Kt}\H_t|\D f|^2,~~\mm-\text {a.e.}
\end{equation}
for any $f\in W^{1,2}$ and $t>0$, where $\H_t f$ is the heat flow from $f$ and $|\D f|$ is the minimal weak upper gradient (or weak gradient for simplicity) of $f$. In particular, by H\"older inequality we know 
\begin{equation}\label{eq:intro-2}
|\D \H_t f|^p \leq e^{-pKt}\H_t|\D f|^p,~~\mm-\text {a.e.}
\end{equation}
for any $p\geq 2$. 
Furthermore, it is proved in \cite{S-S} that inequality \eqref{eq:intro-1} can be improved as
\begin{equation}\label{eq:intro-3}
|\D \H_t f| \leq e^{-Kt}\H_t|\D f|,~~\mm-\text {a.e.}.
\end{equation}
In conclusion, inequality \eqref{eq:intro-2}  holds for any $p\in [1, \infty]$.

\bigskip

Conversely, it is shown in \cite{AGS-B} that a space satisfying inequality \eqref{eq:intro-1} is  $\rcd$.  Let $\ms$ be an infinitesimally Hilbertian space, we have a well-defined Dirichlet energy:
\[
E(f):=\frac12 \int |\D f|^2\,\d \mm
\]
for any $f\in W^{1,2}\ms$.  We denote the $L^2$-gradient flow of $E(\cdot)$ starting from $f$ by $(\H_t f)_t$. Assume further that the space $\ms$ has  Sobolev-to-Lipschitz property: for any function $f\in W^{1,2}$  with$|\D f| \in L^\infty$, we can find a Lipschitz continuous function  $\bar{f}$ such that $f=\bar{f}$ $\mm$-a.e. and  $\Lip (\bar{f})=\esup ~{|\D f|}$. If 
\begin{equation}\label{eq:intro-4}
|\D \H_t f|^2 \leq e^{-2Kt}\H_t|\D f|^2,~~\mm-\text {a.e.}
\end{equation}
for any $f\in W^{1,2}$ and $t>0$, then  $\ms$ is  $\rcd$.

\bigskip

The main goal of this paper is to prove that  for any $p>2$, $p$-gradient estimate  \eqref{eq:intro-2} can also characterize the curvature-dimension condition  of metric measure spaces. We prove a non-smooth version of  $2) \Rightarrow 3)$  in von Renesse-Sturm's result, thus we complete the circle $1)\Leftrightarrow 2) \Leftrightarrow 3)$ in non-smooth setting. 

\bigskip

Now, we introduce our main result in this paper. When $p=2$, it is proved in \cite{S-S} that  there  exists a space of test functions ${\rm TestF}\ms$ which is a dense subspace of $ W^{1,2}(X)$ defined as
\[
{\rm TestF}\ms:= \Big\{f \in {\rm D} ({\bf \Delta}) \cap L^\infty: |\D f|\in L^\infty~~ {\rm and}~~~ \Delta f  \in W^{1,2} \cap L^\infty(X,\mm)\Big\},
\]
such that ${\bf \Delta} |\D f|^2 $ is a well-defined measure (see Definition \ref{def-mlaplacian})  for any $f\in {\rm TestF}$. So it is reasonable to the following assumption ({\bf Assumption} \ref{assumption-2},   see a similar assumption in \cite{S-R}): there exists a dense subspace $\mathcal{A}$ in ${\rm TestF}$  with respect to the graph norm
\[
f\mapsto \Big [ \| (-\Delta)^{\frac 32} f\|^2_{L^2}+\|f\|^2_{W^{1,2}} \Big ]^\frac 12=\Big [  E(\Delta f)+\|f\|^2_{W^{1,2}} \Big ]^\frac12
\]
such that $|\D f|^2 \in \M_\infty$ for any $f\in\mathcal{A}$. We remark that we do not need to assume the density of $\mathcal{A}$ in $W^{1,2}$.

\begin{theorem}[Theorem \ref{th-improvedBE}, Improved Bakry-\'Emery theory]\label{theorem}
Let $M:=\ms$ be a metric measure space such that there exists an algebra $\mathcal{A}$ as described above. If for any $f\in W^{1,2}(X)\cap \Lip(X)\cap L^\infty(X)$ we have the gradient estimate
\begin{equation}\label{eq:th-1}
|\D \H_t f|^p \leq e^{-pKt}\H_t|\D f|^p,~~\mm-\text {a.e.}
\end{equation}
for some $p\in (1, \infty)$. Then \eqref{eq:th-1} holds for $p=1$. In particular, $M$ is a $\rcd$ space.
\end{theorem}

Since we do not have second order differentiation formula for relative entropy along Wasserstein geodesics, or Taylor's expansion in non-smooth setting, we can not simply use the argument in smooth metric measure space (see the proofs in \cite{SVR-T}). The argument we adopt here is the so-called `self-improvement' method in Bakry-\'Emery's $\Gamma$-calculus, which was used in \cite{S-S} to deal with the non-smooth problems.  We remark that we not only use  `self-improvement' technique, but an improved iteration method based on this technique. We believe that this method also has potential application in the future.

\bigskip

It can be seen that Assumption \ref{assumption-2} is satisfied in the following cases, where we can apply our main result.

\underline{ Example 1.} Smooth metric measure space:  obviously,  $C_c^\infty(M)$, the space of smooth functions with compact support is a good algebra in Assumption \ref{assumption-2}. Hence we obtain a new quick proof to von Renesse-Sturm's theorem, without using Taylor's expansion method.

\underline{ Example 2.} $\rcd$ metric measure space: it is proved in Lemma 3.2 \cite{S-S}  that $|\D f|^2\in \M_\infty$ for any $f\in {\rm TestF}$.  By Theorem 1.1 we obtain the following  proposition which deals with the optimal comstant $K$ in the curvature-dimension condition. It is also a complement to Savar\'e's result in \cite{S-S}. 

\begin{proposition}[Self-improvement of gradient estimate]
Let $\ms$ be a $\rcd$ metric measure space. If for any $f\in W^{1,2}\cap \Lip(X)\cap L^\infty(X) $ we have the gradient estimate
\begin{equation}
|\D \H_t f|^p \leq e^{-pK't}\H_t|\D f|^p,~~\mm-\text {a.e.}
\end{equation}
for some $p\in [1, \infty)$ and $K' >K$. Then $\ms$ is a ${\rm RCD}(K', \infty)$ space. In particular, we know
\begin{equation}
|\D \H_t f| \leq e^{-K't}\H_t|\D f|,~~\mm -\text {a.e.}.
\end{equation}
\end{proposition}

\bigskip
In \cite{G-N}, Gigli defines measure valued Ricci tensor on ${\rm RCD}$ metric measure space (see also \cite{H-R}) as
\[
\bRic(\nabla f, \nabla f):={\bf \Gamma}_2(f)-|\Hess[f]|^2_{\rm HS}\,\mm
\]
where ${\bf \Gamma}_2(f):=\frac12 {\bf \Delta} |\D f|^2-\la \nabla f, \nabla \Delta f\ra\,\mm$ and $|\Hess[f]|_{\rm HS}$ is the Hilbert-Schmidt norm of the Hessian $\Hess[f]$ as a module (see \cite{G-N} for details). He shows that $\bRic(\nabla f, \nabla f)\geq K|\D f|^2\,\mm$ if and only if the space is $\rcd$. However, we do not know if $\bRic$ has locality in the sense that $\bRic(\nabla f, \nabla f)\restr{\{|\D f|=0\}}=0$. 

From the proof of Theorem \ref{theorem} we have the following new characterization of curvature bound which extends Gigli's result:
\begin{proposition}[Proposition \ref{prop-last}]
Let $\ms$ be a {\rm RCD} space.  For any $f$ such that $\bRic(\nabla f, \nabla f)$ is well-defined, we denote the Lebesgue decomposition of $\bRic(\nabla f, \nabla f)$ with respect to $\mm$ by
$$\bRic(\nabla f, \nabla f)=\Ric_{ac}(\nabla f, \nabla f)\,\mm+\bRic_{sing}(\nabla f, \nabla f).$$ Then the following characterizations are equivalent. 
\begin{itemize}
\item [1)] $\ms$ is $\rcd$,
\item [2)] for any test function $f\in {\rm TestF}$ we have 
   $\bRic(\nabla f, \nabla f)\geq K|\D f|^2\,\mm$ in the sense that 
   \[
   \Ric_{ac}(\nabla f, \nabla f) \geq  K|\D f|^2~~~\mm-\text{a.e.}
   \]
   and $\bRic_{sing}(\nabla f, \nabla f) \geq 0$,
\item [3)] for any test function $f\in {\rm TestF}$ we have  
$$|\D f|^2  \Ric_{ac}(\nabla f, \nabla f) \geq  K|\D f|^4~~~\mm-\text{a.e.}$$  and $\bRic_{sing}(\nabla f, \nabla f) \geq 0$.
\end{itemize}

\end{proposition}
We remark that this naive extension is non-trivial, because 2) is not a direct consequence of 3) due to lack of the  locality of $\bRic(\cdot, \cdot)$. From this proposition, we know that  $\overline{\Ric}(\nabla f, \nabla f):=|\D f|^2 \Ric_{ac}(\nabla f, \nabla f)\,\mm$ characterizes the Ricci curvature of $\ms$ and ${\overline{\Ric}}$ has locality in the sense that $$\overline{\Ric}(\nabla f, \nabla f)\restr{\{|\D f|=0\}}=0.$$
%%%%%%%%%%%%%%%%%%%%%%%%%
%预备知识
%%%%%%%%%%%%%%%%%%%%%%%%%%%
\section{Preliminaries}  
  First of all, we summarize the  basic hypothesis on the metric measure space $\ms$ below in Assumption \ref{assumption-1} below,   the notions and concepts in in this assumption will be explained later. 
\begin{assumption}\label{assumption-1}
We assume that:
\begin{itemize}
\item [(1)] $(X, \d)$ is a complete, separable geodesic space,
\item [(2)] $\supp \mm=X,~~~~\mm(B_r(x)) < c_1\exp{(c_2 r^2)}~~~\text{for every}~~r>0$,
\item [(3)] $W^{1,2}(X)$ is a Hilbert space,
\item [(4)] $\ms$ has Sobolev-to-Lipschitz property,
\item [(5)] there exits a unique heat kernel $p_t(x, y)$.
\end{itemize}
\end{assumption}

The Sobolev space $W^{1,2}\ms$ is defined as in \cite{AGS-C}. We say that $f\in L^2(X, \mm)$ is a Sobolev function in $W^{1,2}\ms$ if there exists a sequence of Lipschitz functions  $(f_n) \subset L^2$,  such that $f_n \to f$ and $\lip{f_n} \to G$ in $L^2$ for some $G \in L^2(X, \mm)$, where $\lip{f_n} $ is the local Lipschitz constant of $f_n$. It is known that there exists a minimal function $G$ in $\mm$-a.e. sense. We call the minimal $G$ the minimal weak  upper gradient (or weak gradient for simplicity) of the function $f$, and denote it by $|\D f|$. It is known that the locality holds for $|\D f|$, i.e. $|\D f|=|\D g|$ a.e. on the set $\{ f=g\}$. Furthermore, we have the lower semi-continuity: if $\{f_n\}_n\subset W^{1,2}\ms$ is a sequence converging to some $f$ in $\mm$-a.e. sense and  $(|\D f_n|)_n$ is bounded in $L^2(X,\mm)$, then $f\in W^{1,2}\ms$ and
\[
\||\D f|\|_{L^2}\leq \lmti{n}{\infty}\||\D f_n|\|_{L^2}.
\]

We equip $W^{1,2}\ms$ with the norm
\[
\|f\|^2_{W^{1,2}\ms}:=\|f\|^2_{L^2(X,\mm)}+\||\D f|\|^2_{L^2(X,\mm)}.
\]
We say that $\ms$ is an infinitesimally Hilbertian space if $W^{1,2}$ is a Hilbert space (see \cite{AGS-M}, \cite{G-O} for more discussions). 

On an infinitesimally Hilbertian space, we have a natural `carr\'e du champ' operator $\Gamma(\cdot, \cdot): [W^{1,2}\ms]^2 \mapsto L^1\ms$  defined by
\[
\Gamma(f, g):= \frac14 \Big{(}|\D (f+g)|^2-|\D (f-g)|^2\Big{)}.
\]
It can be seen that $\Gamma(\cdot, \cdot)$ is symmetric, bilinear and continuous. We  denote $\Gamma(f,f)$ by $\Gamma(f)$. We  have the following chain rule and Leibnitz rule (Lemma 4.7 and Proposition 4.17 in \cite{AGMR-R}, see also Corollary 7.1.2 in \cite{BH-D})
\[
\Gamma(\Phi(f), g)=\Phi'(f)\Gamma(f,g)~~\text{for every}~f,g\in W^{1,2},~~\Phi\in \Lip{(\R)}, \Phi(0)=0
\]
and
\[
\Gamma(fg,h)=f\Gamma(g,h)+g\Gamma(f,h)~~\text{for every}~f,g,h\in W^{1,2}\cap L^\infty.
\]

We say that   a metric measure space $M=(X ,\d, \mm)$ has  Sobolev-to-Lipschitz property if: for any function $f\in W^{1,2}$  with $|\D f| \in L^\infty$, we can find a Lipschitz continuous function  $\bar{f}$ such that $f=\bar{f}$ $\mm$-a.e. and  $\Lip (\bar{f})=\esup ~{|\D f|}$.

We define the Dirichlet (energy) form $E: L^2 \mapsto [0,\infty]$ by
\[
E(f):=\frac12 \int \Gamma(f)\,\d\mm.
\]
It is proved (see \cite{AGS-C, AGS-D}) that  Lipschitz functions are dense in energy: for any $f\in W^{1,2}$ there is a sequence of Lipschitz functions $(f_n)_n\subset  L^2(X,\mm)$ such that $f_n \rightarrow f$ and $\lip{ f_n} \rightarrow |\D f|$ in $L^2$. Moreover, if $W^{1,2}$ is Hilbert we know Lipschitz functions are dense (strongly) in $W^{1,2}$.

It can be proved that $E$ is a strongly local, symmetric, quasi-regular Dirichlet form (see \cite{AGS-B, AGS-C, AGS-M}). The Markov semigroup $(\H_t)_{t\geq 0}$ generated by $E$ is called the heat flow.  There exists heat kernel which  is a family of functions $p_t(x, y): X\times X \times \R \mapsto \R$ such that $p_t(x, y)\,\d \mm(y)$ is a probability measure for any $x\in X, t\in \R$, and $\H_t f(x)=\int f(y) p_t(x,y)\,\d \mm(y)$  for any $f\in L^2(X, \mm)$.

 For any $f\in L^2(X, \mm)$ we know that $(0, \infty)\ni t \mapsto \H_t f \in L^2\cap D(\Delta)$ satisfies
 \[
\ddt \H_t f=\Delta \H_t f~~\forall t \in (0,\infty),
\]
and
\[
\lmt{t}{0} \H_t f=f~~\text {in}~L^2.
\]
Here the Laplacian is defined  in the following way (see \cite{G-O} for the compatibility of different definitions of Laplacian):
\begin{definition}[Measure valued Laplacian, \cite{G-O, G-N, S-S}]\label{def:laplacian} 

The domain of the Laplacian ${\rm D}({\bf \Delta}) \subset  W ^{1,2}$ consists of $f \in  W ^{1,2} $ such that there is a measure ${\bf \mu}\in {\rm Meas}(M)$ satisfying
\[
\int \varphi \,{\mathbf \mu}= -\int \Gamma (\varphi,  f ) \,   \mm, \forall \varphi: M \mapsto  \R, ~~ \text{Lipschitz with bounded support}.
\]
In this case the measure $\mu$ is unique and we denote it by ${\bf \Delta} f$. If ${\bf \Delta} f \ll m$, we denote its density  with respect to $\mm$ by $\Delta f$.
\end{definition}

We define  ${\rm TestF}\ms \subset W^{1,2}\ms$, the space of test functions  as
\[
{\rm TestF}\ms:= \Big\{f \in {\rm D} ({\bf \Delta}) \cap L^\infty: |\D f|\in L^\infty~~ {\rm and}~~~ \Delta f  \in W^{1,2} \cap L^\infty(X,\mm)\Big\}.
\]
It is known from \cite{S-S} and \cite{AGS-M} that ${\rm TestF}(M)$ is an algebra and it is  dense in $W^{1,2}\ms$ when $\ms$ is a ${\rm RCD}$ metric measure space.  We will see in Lemma \ref{lemma:density}  that 
 ${\rm TestF}$ is  dense in $W^{1,2}$ even when $L^p$-gradient estimate for heat flow holds for some $p>2$.

\begin{lemma}[Lemma 3.2, \cite{S-S}]\label{lemma:cv}
Let $M=\ms$ be a metric measure space satisfying Assumptions \ref{assumption-1}. Assume that the algebra generated by $\{f_1, ... , f_n \}\subset {\rm TestF}(M)$ is included in ${\rm TestF}(M)$.  Let $\Phi \in C^\infty(\R^n)$ with $\Phi(0)=0$. Put $ {\bf f}=(f_1, ... , f_n)$,
then $\Phi({\bf f}) \in {\rm TestF}(M)$. 
\end{lemma}

Let $f \in {\rm TestF}(M)$. We define the Hessian $\Hess[f](\cdot, \cdot): \{{\rm TestF}(M)\}^2 \mapsto L^0(M)$ by
\[
2\Hess{[f]}( g,h)=\Gamma ( g,  \Gamma ( f,  h )) +\Gamma(h,  \Gamma(f, g))-\Gamma( f, \Gamma( g, h)).
\]
We have the following lemma.

\begin{lemma}[Chain rules, \cite{B-H}, \cite{S-S}]\label{lemma:chain}
Let $f_1, ... , f_n \in {\rm TestF}(M)$ and $\Phi \in C^\infty(\R^n)$  with $\Phi(0)=0$. Assume that the algebra generated by $\{f_1, ... , f_n \}\subset {\rm TestF}(M)$ is included in ${\rm TestF}(M)$. Put $ {\bf f}=(f_1, ... , f_n)$, then
\[
|\D \Phi({\bf f})|^2\,\mm=\mathop{\sum}_{i,j=1}^n \Phi_i\Phi_j ({\bf f}) \Gamma (f_i,  f_j )\,\mm,
\]
and
\[
{\bf \Delta} \Phi({\bf f})=\mathop{\sum}_{i=1}^n \Phi_i({\bf f}){\bf \Delta} f_i+\mathop{\sum}_{i,j=1}^n \Phi_{ij}({\bf f}) \Gamma  (f_i,  f_j ) \,\mm.
\]

\end{lemma}

The last lemma will be used in the proof of Theorem \ref{th-improvedBE}.

\begin{lemma}[Lemma 3.3.6, \cite{G-N}]\label{lemma:delta}
Let $\mu_i=\rho_i \,\mm+\mu_i^s$  be measures with $\mu_i^s \perp \mm$, $i=1,2,3$.  We assume that
\[
\lambda^2 \mu_1+2\lambda \mu_2 + \mu_3 \geq 0, ~~~~~\forall \lambda \in \R.
\]
Then we have
\[
\mu_1^s \geq 0,~~~\mu_3^s \geq 0
\]
and
\[
|\rho_2|^2 \leq \rho_1 \rho_3,~~~\mm-\text {a.e.}.
\]
\end{lemma}

\section{Main Results}
Firstly, we  discuss more about the measure-valued Laplacian.
Since $E$ is quasi-regular,  we know (see Remark 1.3.9 (ii), \cite{CF-S}) that every function $f\in W^{1,2}$ has an
quasi-continuous representative $\overline{f}$. And $\overline{f}$ is unique up to quasi-everywhere equality, i.e. if $\tilde{f}$ is another quasi-continuous representative, then $\tilde{f}=\overline{f}$ holds in a complement of an $E$-polar set. For more details,  see Definition 2.1 in \cite{S-S} and the references therein.

%The first lemma characterize the measure valued Laplacian, which can be seen as a equivalent definition of %Lapalcian (see Definition \ref{def:laplacian}).

\begin{definition}\label{def-mlaplacian}
We define $\M_\infty$ the space of  $f\in {\rm D}({\bf \Delta})\cap L^\infty$ such  that there exists a measure decomposition ${\bf \Delta} f=\mu_+-\mu_-$ with $\mu_{\pm}$  in the positive cone in $(W^{1,2})'$, such that:
\[
\int \overline{\varphi} \,\d ({\bf \Delta}f)=-\int \Gamma(\varphi, f)\,\d \mm
\]
for any $\varphi\in W^{1,2}$ and the quasi-continuous representative $\overline{\varphi}\in L^1(X, {\bf \Delta}f)$.

In particular, every $E$-polar set is $({\bf \Delta} f)$-negligible and the measure  $\overline{\varphi}{\bf \Delta}f$ is well-defined.
\end{definition}

\bigskip

In the next lemma we study the measure ${\bf \Delta} \Gamma (f)^{\frac{p}2}$. Since $\Gamma (f)$ is not necessarily continuous, and $\Phi(x)=x^{\frac{p}2}$ is not $C^2(\R)$, we can not use Lemma \ref{lemma:chain} directly.

\begin{lemma}\label{lemma-1}
Let $\ms$ be a metric measure space satisfying Assumptions \ref{assumption-1}.
Let $f\in {\rm TestF}$ such that $\Gamma(f), \Gamma(f)^{\frac{p}2} \in \M_\infty$, $p> 2$. Then
\begin{equation}\label{eq:lemma1-1}
\frac1{p}{\bf \Delta} \Gamma (f)^{\frac{p}2}-\Gamma( f)^{\frac{p}2-1} \Gamma(\Delta f, f)\d \mm\geq K\Gamma(f)^{\frac{p}2}\d\mm
\end{equation}
if and only if
\begin{equation}\label{eq:lemma1-2}
\frac{1}2\Gamma(f){\bf \Delta }_{ac}\Gamma(f)+\frac12(\frac{p}2-1) \Gamma(\Gamma(f)) \d\mm\geq  \Big (\Gamma(f)\Gamma(\Delta f, f)+K\Gamma(f)^2\Big) \d\mm
\end{equation} 
and $\overline{\Gamma(f)}{\bf \Delta }_{sing}\Gamma(f)\geq 0$ as measures, where ${\bf \Delta }_{ac}\Gamma(f)$ is the absolutely continuous part in the measure decomposition ${\bf \Delta }\Gamma(f)={\bf \Delta }_{ac}\Gamma(f)+{\bf \Delta }_{sing}\Gamma(f)$ with respect to $\mm$,  and  $\overline{\Gamma(f)}$ is the quasi-continuous representation of $\Gamma(f)$.
\end{lemma}
\begin{proof}
Since $p>2$, it can be seen that \eqref{eq:lemma1-2} is equivalent to
\begin{equation}\label{eq:lemma1-2.4}
\frac{1}2\Gamma(f)^{\frac{p}2-1}{\bf \Delta }_{ac}\Gamma(f)+\frac12(\frac{p}2-1)\Gamma(f)^{\frac{p}2-2} \Gamma(\Gamma(f)) \d\mm\geq  \Big (\Gamma(f)^{\frac{p}2-1}\Gamma(\Delta f, f)+K\Gamma(f)^{\frac{p}2}\Big) \d\mm.
\end{equation}
Assume that we  have the decomposition of the measure $\frac1{p}{\bf \Delta} \Gamma (f)^{\frac{p}2}$ with respect to $\mm$: $
\frac1{p}{\bf \Delta} \Gamma (f)^{\frac{p}2}=\frac1{p}{\bf \Delta}_{sing} \Gamma (f)^{\frac{p}2}+\frac1{p}{ \bf \Delta}_{ac} \Gamma (f)^{\frac{p}2}$.
From \eqref{eq:lemma1-1} we know the singular part $\frac1{p}{\bf \Delta}_{sing} \Gamma (f)^{\frac{p}2}$ of the measure $\frac1{p}{\bf \Delta} \Gamma (f)^{\frac{p}2}$ is non-negative.

From hypothesis we know $\Gamma(f), \Gamma (f)^{\frac{p}2} \in D({\bf \Delta})$, by chain rule we know
\begin{equation}\label{eq:lemma1-2.5}
\int \varphi \,\d{\bf \Delta} \Gamma (f)^{\frac{p}2}=-\int \Gamma(\varphi,  \Gamma (f)^{\frac{p}2})\, \d \mm=-\int \frac{p}2 \Gamma(f)^{\frac{p}2-1}\Gamma(\varphi, \Gamma(f))\,\d \mm
\end{equation}
for any Lipschitz function $\varphi$ with bounded support.

Denote by $\overline{\Gamma(f)}$  the quasi-continuous representation of $\Gamma(f)$. From Leibniz rule and chain rule we know $\varphi ( \Gamma (f)+\epsilon)^{\frac{p}2-1}\in W^{1,2}$, for any $\epsilon>0$.  According to Definition \ref{def-mlaplacian} we have
\begin{eqnarray*}
&&-\int \varphi (\overline{\Gamma(f)}+\epsilon)^{\frac{p}2-1}\,\d{\bf \Delta }\Gamma(f) =\int \Gamma(\varphi ({\Gamma(f)}+\epsilon)^{\frac{p}2-1}, \Gamma(f))\,\d \mm \\
&=& \int \varphi (\frac{p}2-1) ({\Gamma(f)}+\epsilon)^{\frac{p}2-2} \Gamma(\Gamma(f))\,\d \mm+\int({\Gamma(f)}+\epsilon)^{\frac{p}2-1}\Gamma(\varphi, \Gamma(f))\,\d \mm.
\end{eqnarray*}
Letting $\epsilon \to 0$, by monotone convergence theorem we obtain
\begin{equation}\label{eq:lemma1-3}
-\int \varphi \overline{\Gamma(f)}^{\frac{p}2-1}\,\d{\bf \Delta }\Gamma(f)  = \int \Big [ \varphi (\frac{p}2-1) {\Gamma(f)}^{\frac{p}2-2} \Gamma(\Gamma(f))+{\Gamma(f)}^{\frac{p}2-1}\Gamma(\varphi, \Gamma(f))\Big ]\,\d \mm.
\end{equation}

Combining \eqref{eq:lemma1-2.5} and \eqref{eq:lemma1-3} we have 
\begin{equation}\label{eq:lemma1-4}
\frac1{p}{\bf \Delta} \Gamma (f)^{\frac{p}2}= \frac{1}2\overline{\Gamma(f)}^{\frac{p}2-1}{\bf \Delta }\Gamma(f)+\frac12(\frac{p}2-1)\Gamma(f)^{\frac{p}2-2} \Gamma(\Gamma(f))  \d\mm
\end{equation}
as measures. Therefore,  we know
\begin{eqnarray*}
\frac1{p}{\bf \Delta}_{ac} \Gamma (f)^{\frac{p}2}&=&\frac{1}2\overline{\Gamma(f)}^{\frac{p}2-1}{\bf \Delta }_{ac}\Gamma(f)+\frac12(\frac{p}2-1)\Gamma(f)^{\frac{p}2-2} \Gamma(\Gamma(f))  \d\mm\\
&=&\frac{1}2{\Gamma(f)}^{\frac{p}2-1}{\bf \Delta }_{ac}\Gamma(f)+\frac12(\frac{p}2-1)\Gamma(f)^{\frac{p}2-2} \Gamma(\Gamma(f))  \d\mm
\end{eqnarray*}
and
\begin{eqnarray*}
\frac1{p}{\bf \Delta}_{sing}\Gamma (f)^{\frac{p}2}=\frac{1}2\overline{\Gamma(f)}^{\frac{p}2-1}{\bf \Delta }_{sing}\Gamma(f).
\end{eqnarray*}

In conclusion,  we obtain
\begin{equation*}
\frac1{p}{\bf \Delta} \Gamma (f)^{\frac{p}2}= \frac{1}2{\Gamma(f)}^{\frac{p}2-1}{\bf \Delta }_{ac}\Gamma(f)+\frac12(\frac{p}2-1)\Gamma(f)^{\frac{p}2-2} \Gamma(\Gamma(f))  \d\mm+\frac{1}2\overline{\Gamma(f)}^{\frac{p}2-1}{\bf \Delta }_{sing}\Gamma(f).
\end{equation*}
Hence \eqref{eq:lemma1-1} is equivalent to  \eqref{eq:lemma1-2.4}, we prove the lemma.
\end{proof}
\bigskip

The following lemma will be used in the proof of Theorem \ref{th-improvedBE}.
\begin{lemma}\label{lemma-2}
Let $P(r):[0, \infty) \mapsto [-\frac14, \infty)$ be a function defined as
 \[
P(r)=r-\frac1{4(r+1)},
\]
and $a_0\geq 0$ be an arbitrary initial datum, we define  $(a_n)_{n\in \N}$ recursively by the formula
\[
a_{n+1}=P(a_n).
\]
Then there exists an integer $N_0$ such that $0\leq a_{N_0}<1$ and $-\frac14 \leq a_{N_0+1}<0$.

 Conversely, for any $a \in [0,1)$ and $b>a$, there exists a sequence $a_0, ..., a_{N_0}$ defined by the recursive function $P$ such that $a_0>b$ and $a_{N_0}=a$.
\end{lemma}
\begin{proof}
It can be seen that $a_{n+1} <a_n$. If $a_0\geq 0$, by monotonicity we know $a_n-a_{n+1} \in [\frac1{4(a_0+1)},\frac14]$ for any $n \in \N$. So there exists a unique $N_0$ such that $0\leq a_{N_0}<1$ and $-\frac14 \leq a_{N_0+1}<0$. Conversely, since $P(r)$ is strictly monotone on $[0, \infty)$, we know $P^{-1}(r):[-\frac14, \infty) \mapsto [0, \infty)$ is well defined. And $(P^{-1})^{(n+1)}(a)-(P^{-1})^{(n)}(a) \in [\frac1{4((P^{-1})^{(n+1)}(a)+1)},\frac14]$ for any $n \in \N$. Thus there exists $N \in \N$ such that $(P^{-1})^{(N_0)}(a) \geq b$. Finally, we can pick $a_0=(P^{-1})^{(N_0)}(a)$, so that $a_{N_0}=(P)^{(N_0)}(a_0)=a$  fulfils our request.
\end{proof}

\bigskip

%可以加上一个说明, 与p=2的情况进行对比（其实是一样的）

As we mentioned in the Introduction, the space of test functions is dense in $W^{1,2}\ms$ when $L^p$-gradient estimate for heat flow holds.

\begin{lemma}[Density of test functions in $W^{1,2}\ms$,  Remark 2.5 \cite{AGS-B}]\label{lemma:density}
Let $\ms$ be a metric measure space satisfying Assumption \ref{assumption-1}.  Assume that for any $f\in W^{1,2}\cap \Lip \cap L^\infty\ms$ we have the $L^p$-gradient estimate
\begin{equation}\label{eq:lemma3-1}
|\D \H_t f|^p \leq e^{-pKt}\H_t|\D f|^p~~\mm-\text {a.e.}
\end{equation}
for some $p\in [1, \infty)$. Then the space of test functions ${\rm TestF}\ms$ is dense in $W^{1,2}$.
\end{lemma}
\begin{proof}
As we discussed in the preliminary section, the space
\[
\V^1:=\Big\{ \varphi \in W^{1,2}: \Gamma(\varphi) \in L^\infty(X, \mm)\Big \}
\]
is dense in $W^{1,2}$. We also know that the 
\[
\V^1_\infty:=\Big\{ \varphi \in W^{1,2}\cap L^\infty: \Gamma(\varphi) \in L^\infty(X, \mm)\Big \}
\]
in dense in $L^2$, and $\V^1_\infty$ is invariant under the action $(\H_t)_t$ by \eqref{eq:lemma3-1} and Sobolev-to-Lipschitz property. Hence by an approximation argument (see e.g. Lemma 4.9 in \cite{AGS-M}), we know $\V^1_\infty$ is dense in $W^{1,2}$. Similarly, by a semigroup mollification (see e.g. page 351, \cite{AGS-B}) we can prove that
\[
\V^2_\infty:=\Big\{ \varphi \in \V^1_\infty: \Delta \varphi \in W^{1,2} \cap L^\infty(X, \mm)\Big \}
\]
is dense in $W^{1,2}$.
\end{proof}

We now introduce the following technical assumption, which  is important in our proof. It can be proved that Riemannian manifolds and $\rcd$ spaces satisfy this assumption.
\begin{assumption}[Existence of good algebra]\label{assumption-2}
We  assume the existence of a dense subspace $\mathcal{A}$ in ${\rm TestF}\ms$ with respect to the graph norm
\[
f\mapsto \Big [ \| (-\Delta)^{\frac 32} f\|^2_{L^2}+\|f\|^2_{W^{1,2}} \Big ]^\frac 12=\Big [  \|\Gamma(\Delta f)\|^2_{L^2}+\|f\|^2_{W^{1,2}} \Big ]^\frac12
\]
such that $\Gamma(f) \in \M_\infty$ for any $f\in\mathcal{A}$. 
\end{assumption} 
It can be seen that  $\mathcal{A}$ is an algebra (i.e.  $\mathcal{A}$ is closed w.r.t. pointwise multiplication), if it is non-trivial. In particular, by Lemma \ref{lemma:density} we know that $\mathcal{A}$ is  dense  in $W^{1,2}$ if $L^p$ gradient estimate holds.
\bigskip

\begin{theorem}[Improved Bakry-\'Emery theory]\label{th-improvedBE}
Let $\ms$ be a metric measure space satisfying Assumption \ref{assumption-1} and Assumption \ref{assumption-2}. If for any $f\in W^{1,2} \cap \Lip \cap L^\infty\ms$ we have the gradient estimate
\begin{equation}\label{eq:th1-1}
|\D \H_t f|^p \leq e^{-pKt}\H_t|\D f|^p,~~\mm-\text {a.e.}
\end{equation}
for some $p\in [1, \infty)$. Then $\ms$ is a $\rcd$ space.
\end{theorem}
\begin{proof}
If $p\leq 2$, by the result in \cite{AGS-B} we know $\ms$ is a $\rcd$. Now we assume $p>2$.

{\bf Part 1.} Firstly, we prove 
\begin{equation}\label{eq:th1-1.5}
\Gamma(f){\bf \Delta }_{ac}\Gamma(f)+\epsilon\Gamma(\Gamma(f)) \geq  \Gamma(f)\Gamma(\Delta f, f)+K\Gamma(f)^2,
\end{equation}
and $\overline{\Gamma(f)}{\bf \Delta }_{sing}\Gamma(f)\geq 0$, for any $f\in \mathcal{A}$ and $\epsilon>0$.

For any $f\in \mathcal{A}, \varphi \in {\rm TestF}\ms, \varphi\geq 0$  and $t>0$, we define $F:[0,t] \mapsto \R$ by
\[
F(s)=\int  e^{-pKs}\H_s\varphi\Gamma(\H_{t-s} f)^{\frac{p}2}.
\]
It can be seen that $F$ is a $C^1$ function (see Lemma 2.1, \cite{AGS-B}). From \eqref{eq:th1-1} we know $F(s) \leq F(t)$ holds for any $s\in [0, t]$. Hence $F'(s)\restr{s=t} \geq 0$, and so
\begin{eqnarray*}
&&\int e^{-pKs}\Delta \H_s \varphi \Gamma(\H_{t-s} f)^{\frac{p}2}\restr{s=t}-p\int e^{-pKs}\H_s \varphi \Gamma(\H_{t-s} f)^{\frac{p}2-1} \Gamma(\Delta \H_{t-s} f, \H_{t-s}f)\restr{s=t}\\
&\geq& pK\int e^{-pKs}\H_s\varphi \Gamma(\H_{t-s} f)^{\frac{p}2} \restr{s=t}.
\end{eqnarray*}

Letting $t \to 0$ we obtain
\[
\int \Delta \varphi \Gamma (f)^{\frac{p}2}-p\int \varphi \Gamma( f)^{\frac{p}2-1} \Gamma(\Delta f, f)\geq pK\int \varphi \Gamma(f)^{\frac{p}2}.
\]
In particular, from Lemma 2.6 and Lemma 3.2 in \cite{S-S}  we know $\Gamma(f)^{\frac{p}2} \in {\rm D}({\bf \Delta})$ and
\begin{equation}\label{eq:th1-2}
\frac1{p}{\bf \Delta} \Gamma (f)^{\frac{p}2}-\Gamma( f)^{\frac{p}2-1} \Gamma(\Delta f, f)\d \mm\geq K\Gamma(f)^{\frac{p}2}\d\mm.
\end{equation}

By Lemma \ref{lemma-1},   we get that
\begin{equation}\label{eq:th1-3}
\frac12\Gamma(f){ \Delta }_{ac}\Gamma(f)+(\frac{p}4-\frac12) \Gamma(\Gamma(f)) \geq  \Gamma(f)\Gamma(\Delta f, f)+K\Gamma(f)^2
\end{equation}
holds $\mm$-a.e., and $\overline{\Gamma(f)}{\bf \Delta }_{sing}\Gamma(f)\geq 0$.

From now on, all the  inequalities are considered in $\mm$-a.e. sense.
We denote  $\frac12{ \Delta }_{ac}\Gamma(f)- \Gamma(\Delta f, f)$ by $\Gamma_2(f)$, and  $\frac12{ \Delta }_{ac}\Gamma(f)- \Gamma(\Delta f, f)-K\Gamma(f)$ by $\Gamma_{2,K}(f)$, then \eqref{eq:th1-3} becomes 
\[
\Gamma_{2,K}(f)\Gamma(f)+(\frac{p}4-\frac12) \Gamma(\Gamma(f)) \geq 0.
\]
For any real number $r\geq 0$, we say that the property $B(r)$ holds if 
\[
\Gamma_{2,K,r}(f):=\Gamma_{2,K}(f)\Gamma(f)+r \Gamma(\Gamma(f)) \geq 0
\]
for any $f\in {\rm TestF}$.
For example, \eqref{eq:th1-3} means $B(\frac{p}4 -\frac12)$.

Now we define 
\[
P(r)=r-\frac1{4(r+1)}.
\]
Then we will prove that  $B(r)$ implies $B(P(r))$. 
We choose the smooth function $\Phi:\R^3 \mapsto \R$ defined by
\[
\Phi({\bf f}):=\lambda f_1+(f_2-a)(f_3-b)-ab,~~a,b,\lambda \in \R.
\]
Then we know 
\begin{eqnarray*}
\Phi_{23}({\bf f})&=&\Phi_{32}= a, ~~~~~~~~~~~~~~~~\Phi_{ij}({\bf f})= 0,~~~~\text {if}~(i,j)\notin \{(2,3), (3,2)\}\\
\Phi_1({\bf f}) &=& \lambda,~~~~~~~~\Phi_2({\bf f}) = f_3-b,~~~~~~~~~~\Phi_3({\bf f}) =f_2-a.
\end{eqnarray*}

If ${\bf f}:=(f,g,h) \in {\mathcal{A}}^3$, we know $\Phi({\bf f}) \in {\mathcal{A}}$ by Lemma \ref{lemma:cv}. Hence we know 
\begin{equation}\label{eq:th1-4}
\Gamma_{2,K}(\Phi({\bf f}))\Gamma(\Phi({\bf f}))+r \Gamma(\Gamma(\Phi({\bf f}))) \geq 0.
\end{equation}

By direct computation using Lemma \ref{lemma:chain} (see also Theorem 3.4, \cite{S-S}), we have
\begin{eqnarray*}
\Gamma(\Phi({\bf f}))&=& g^{ij}\Phi_i\Phi_j({\bf f})\\
&=& \lambda^2 \Gamma(f)+ (g-a)A_1+(h-b)B_1
\end{eqnarray*}
where $g^{ij}=\Gamma(f_i,f_j)$, $A_1, A_2$ are some additional terms.

Similarly, we have
\begin{eqnarray*}
\Gamma(\Gamma(\Phi({\bf f})))&=& \Gamma(g^{ij}\Phi_i\Phi_j({\bf f}))\\
&=& (g^{ij})^2\Gamma(\Phi_i\Phi_j)+(\Phi_i\Phi_j)^2\Gamma(g^{ij})+2g^{ij}\Phi_i\Phi_j\Gamma(g^{ij},\Phi_i\Phi_j)\\
&=& (g^{ij})^2\Big[  \Phi_i^2\Gamma(\Phi_j)+\Phi_j^2\Gamma(\Phi_i)+2\Phi_i\Phi_j\Gamma(\Phi_i, \Phi_j)\Big]\\
&&~+(\Phi_i\Phi_j)^2\Gamma(g^{ij})+2g^{ij}\Phi_i\Phi_j\Gamma(g^{ij}, \Phi_i\Phi_j)\\
&=&2(g^{12})^2\lambda^2\Gamma(h)+2(g^{13})^2\lambda^2\Gamma(g)+\lambda^4\Gamma(g^{11})+(g-a)A_2+(h-b)B_2\\
&=&2\Gamma(f,g)^2\lambda^2\Gamma(h)+2\Gamma(f,h)^2\lambda^2\Gamma(g)+\lambda^4\Gamma(\Gamma(f))+(g-a)A_2+(h-b)B_2.
\end{eqnarray*}

We also know (see Theorem 3.4, \cite{S-S} or Lemma 3.3.7, \cite{G-N}) that
\begin{eqnarray*}
\Gamma_2({\bf f})-K\Gamma(\Phi({\bf f}))&=& \lambda^2\Gamma_2(f)+4\lambda\Hess [f](g,h)+2\Big(\Gamma(g)\Gamma(h)+\Gamma(g,h)^2 \Big)\\
&+& (g-a)A_3+(h-b)B_3 -K\lambda^2 \Gamma(f).
\end{eqnarray*}

Combining the computations above, \eqref{eq:th1-4} becomes an  inequality with parameters $a, b, \lambda$. By locality of weak gradients and density of simple functions, we can replace $b$  by $h$ and replace  $a$ by $g$ (similar arguments are used in Theorem 3.4 \cite{S-S} and Lemma 3.3.7 \cite{G-N}). Then we obtain the following inequality from \eqref{eq:th1-4}
\begin{eqnarray*}
&&\lambda^2 \Gamma(f)\left[\lambda^2\Gamma_2(f)+4\lambda\Hess [f](g,h)+2\Big(\Gamma(g)\Gamma(h)+\Gamma(g,h)^2 \Big)  -K\lambda^2 \Gamma(f)\right]\\&+&
r\left[2\Gamma(f,g)^2\lambda^2\Gamma(h)+2\Gamma(f,h)^2\lambda^2\Gamma(g)+\lambda^4\Gamma(\Gamma(f))\right]\\
&\geq&0.
\end{eqnarray*}
Since $r \geq 0$ and 
\[
\Gamma(g)\Gamma(h) \geq \Gamma(g,h)^2,
\]
we know
\begin{eqnarray*}
&&\Gamma(f)\left[\lambda^2\Gamma_2(f)+4\lambda\Hess [f](g,h)+4\big(\Gamma(g)\Gamma(h) \big)  -K\lambda^2 \Gamma(f)\right]\\&+&
r\left[4\Gamma(f)\Gamma(g)\Gamma(h)+\lambda^2\Gamma(\Gamma(f))\right]\\
&\geq&0.
\end{eqnarray*}
Then we have
\[
(\Gamma_2(f)\Gamma(f)+r\Gamma(\Gamma(f))-K\Gamma(f)^2)\lambda^2+4\lambda\Gamma(f)\Hess [f](g,h)+4(r+1)\Gamma(f)\Gamma(g)\Gamma(h) \geq 0.
\]

Applying Lemma \ref{lemma:delta} we obtain
\[
(1+r)\Gamma_{2,K,r}\Gamma(f)\Gamma(g)\Gamma(h) \geq \Gamma(f)^2\Hess [f](g,h).
\]
Since $B(r)$ means $\Gamma_{2,K,r} \geq 0$, this inequality is equivalent to
\begin{equation}\label{eq:th1-5}
(1+r)\Gamma_{2,K,r}(f)\Gamma(g)\Gamma(h) \geq \Gamma(f)\Hess [f](g,h).
\end{equation}

Recall that $2\Hess{[f]}( g,h)=\Gamma ( g,  \Gamma ( f,  h )) +\Gamma(h,  \Gamma(f, g))-\Gamma( f, \Gamma( g, h))$, we know
\[
\Hess[f](g,h)+\Hess[g](f,h)=\Gamma(\Gamma(f,g),h).
\]
Combining with inequality \eqref{eq:th1-5} we have
\begin{eqnarray*}
\sqrt{\frac{1}{1+r}} \Gamma(\Gamma(f,g),h) \sqrt{\Gamma(f)} &\leq& \sqrt{\Gamma_{2,K,r}(f)\Gamma(g)\Gamma(h) }+\sqrt{\Gamma_{2,K,r}(g)\Gamma(f)\Gamma(h)}\\
&=& \Big ( \sqrt{\Gamma_{2,K,r}(f)\Gamma(g) }+\sqrt{\Gamma_{2,K,r}(g)\Gamma(f) }\Big)\sqrt{\Gamma(h)}.
\end{eqnarray*}

Then we fix $f, g \in {\mathcal{A}}$, and approximate any $h \in W^{1,2} \cap L^\infty$ with a sequence $(h_n) \subset {\mathcal{A}}$ converging to $h$ strongly in $W^{1,2}$ such that 
\[
\Gamma(h_n) \to \Gamma(h), ~~~\Gamma(h_n, \Gamma(f,g)) \to \Gamma(h, \Gamma(f,g))
\]
pointwise and in $L^1(X, \mm)$. Thus we can replace $h$ by $\Gamma(f,g)$ in the last inequality and obtain
\begin{equation}\label{eq:th1-6}
\sqrt{\frac{1}{1+r}}\sqrt{ \Gamma(\Gamma(f,g)) \Gamma(f)}=\Big ( \sqrt{\Gamma_{2,K,r}(f)\Gamma(g) }+\sqrt{\Gamma_{2,K,r}(g)\Gamma(f) }\Big).
\end{equation}

Let $g=f$ in \eqref{eq:th1-6}, we obtain
\[
\frac1{1+r}\Gamma(\Gamma(f))\Gamma(f) \leq 4\Gamma_{2,K,r}(f) \Gamma(f).
\]
Therefore,
\[
(\frac14\frac1{1+r}-r)\Gamma(\Gamma(f))\Gamma(f) \leq \Gamma_{2,K}(f) \Gamma(f).
\]
In other words, we have $B(P(r))$.

From Lemma \ref{lemma-2} we know there exists $a_0 \geq \frac{p}4-\frac12$ and $N_0 \in \N$ such that $a_{N_0}=\epsilon$, where $a_{n+1}=P(a_n)$, $n=0,...,N_0-1$. Then
 we know $B(a_0)$ from \eqref{eq:th1-3}. From  the result above, we can see that $B(a_{N_0})$ holds by induction. So we prove \eqref{eq:th1-1.5}.

 {\bf Part 2.}
From \eqref{eq:th1-1.5} and Lemma \ref{lemma-1} we know 
 \begin{equation}\label{eq:th1-6.5}
 \frac1{p_n}{\bf \Delta} \Gamma (f)^{\frac{p_n}2}-\Gamma( f)^{\frac{p_n}2-1} \Gamma(\Delta f, f)\, \mm\geq K\Gamma(f)^{\frac{p_n}2}\,\mm
 \end{equation}
for any $p_n=2+\frac{1}{2^{n}}$,   $n\in \N$.

 Let  $f\in \mathcal{A}$,  $\varphi\in {\rm TestF}$ and $\varphi \geq 0$.  From \eqref{eq:th1-6.5} we know
 \[
 \int \frac1{p_n} \Delta\varphi \Gamma (f)^{\frac{p_n}2}\,\d\mm-\int \varphi\Gamma( f)^{\frac{p_n}2-1} \Gamma(\Delta f, f)\d \mm\geq K\int\varphi\Gamma(f)^{\frac{p_n}2}\d\mm.
 \]

Letting $n\to \infty$,  by dominated convergence theorem and monotone convergence theorem we know 
\begin{equation}\label{eq:th1-7}
\frac 12 \int  \Delta\varphi \Gamma (f)\,\d\mm-\int \varphi \Gamma(\Delta f, f)\,\d \mm \geq K\int\varphi\Gamma(f)\,\d\mm.
\end{equation}

Combining with the density of $\mathcal{A}$ in ${\rm TestF}$, we know \eqref{eq:th1-7} holds for all $f\in {\rm TestF}$.

Finally, by Theorem 4.17  \cite{AGS-B} we know that $\ms$ is a $\rcd$ space.
\end{proof}

\bigskip

As a corollary, we have the following proposition. We recall (see \cite{G-N}) that the measure-valued Ricci tensor on ${\rm RCD}$ metric measure space is defined as
\[
\bRic(\nabla f, \nabla f):={\bf \Gamma}_2(f)-|\Hess[f]|^2_{\rm HS}\,\mm,
\]
where ${\bf \Gamma}_2(f):=\frac12 {\bf \Delta} \Gamma(f)-\Gamma( f,  \Delta f)\,\mm$ and $|\Hess[f]|_{\rm HS}$ is the minimal $L^2$ function $G$ such that $|\sum_{i,j}\Hess[f](g_i,h_j)|\leq G\sqrt{\sum_{i,j}\Gamma^2(g_i, h_j)}$ for any $(g_i), (h_j) \subset {\rm TestF}$ (see \cite{G-N} and \cite{S-S} for details). It is proved that $\bRic$ is well defined for any $f\in {\rm TestF}\ms$ when $\ms$ is ${\rm RCD}$.
\begin{proposition}\label{prop-last}
Let $\ms$ be a {\rm RCD} space. Then the following characterizations are equivalent.
\begin{itemize}
\item [1)] $\ms$ is $\rcd$,
\item [2)] for any test function $f\in {\rm TestF}$ we have 
   $\bRic(\nabla f, \nabla f)\geq K|\D f|^2\,\mm$ in the sense that 
   \[
   \Ric_{ac}(\nabla f, \nabla f) \geq  K|\D f|^2~~~\mm-\text{a.e.}
   \]
   and $\bRic_{sing}(\nabla f, \nabla f) \geq 0$.
\item [3)] for any test function $f\in {\rm TestF}$ we have  
$$|\D f|^2  \Ric_{ac}(\nabla f, \nabla f) \geq  K|\D f|^4~~~\mm-\text{a.e.}$$  and $\bRic_{sing}(\nabla f, \nabla f) \geq 0$.
\end{itemize}

\end{proposition}

\begin{proof}
1) $\Rightarrow$ 2) is Lemma 3.6.2  \cite{G-N}, 2) $\Rightarrow$ 3) is trivial. So we just need to prove 3) $\Rightarrow$ 1). 

From 3) we know $\Gamma_{2,K,0}(f) \geq 0$, $\mm$-a.e. for any $f\in {\rm TestF}$. Therefore $\Gamma_{2,K,r}(f) \geq 0$  for any $r>0$. Using the same argument as   in the proof of Theorem \ref{th-improvedBE}, we know $\ms$ is $\rcd$.
\end{proof}

%%%%%%%%%%%%%%%%%%%%%%%%%%%%%%%%%%%%%%%%%%%%% 参考文献
%%%%%%%%%%%%%%%%%%%%%%%%%%%%%%%%%%%%%%%%%%%%%%%%%%%%

\def\cprime{$'$}

\bigskip

Bang-Xian Han, Institute for applied mathematics, University of Bonn

 Endenicher Allee 60 , D-53115 Bonn, Germany
 
{ Email}: han@iam.uni-bonn.de

\end{document}